\documentclass[10pt]{article}

\usepackage[nottoc,notlot,notlof]{tocbibind}
\usepackage{amsmath,amssymb,amsfonts,amsthm}
\usepackage{latexsym}
\usepackage{tikz}
\usepackage[normalem]{ulem}
\usetikzlibrary{automata,positioning}
\usetikzlibrary{chains,fit,shapes}
\usetikzlibrary{calc}
\usetikzlibrary{arrows}

\usetikzlibrary{positioning,calc}
\usetikzlibrary{graphs}
\usetikzlibrary{graphs.standard}
\usetikzlibrary{arrows,decorations.markings}
\usepackage{multicol}
\usepackage{xcolor}
\newtheorem{theorem}{Theorem}[section]
\newtheorem{corollary}{Corollary}[section]
\newtheorem{lemma}{Lemma}[section]
\newtheorem{prop}{Proposition}[section]
\theoremstyle{definition}
\newtheorem{definition}{Definition}[section]
\newtheorem{observation}{Observation}[section]
\newtheorem{example}{Example}[section]

\numberwithin{equation}{section}
\title{Square-free Word-representation of Word-representable Graphs}

\author{\hspace{1cm} Biswajit Das \ and Ramesh Hariharasubramanian \\ 
{{\footnotesize d.biswajit@iitg.ac.in},\ {\footnotesize  ramesh\_h@iitg.ac.in}}\\{\footnotesize Department of Mathematics, Indian Institute of Technology Guwahati, Guwahati, Assam 781039, India}}

\begin{document}
	\maketitle
	
	\begin{abstract}
	 A graph $G = (V, E)$ is \textit{word-representable}, if there exists a word $w$ over the alphabet $V$ such that for letters $\{x,y\}\in V$, $x$ and $y$ alternate in $w$ if and only if $xy \in E$. 

    In this paper, we prove that any non-empty word-representable graph can be represented by a word containing no non-trivial squares. This result provides a positive answer to the open problem present in the book \textit{Words and graphs} written by \textit{Sergey Kitaev}, and \textit{Vadim Lozin}. Also, we prove that for a word-representable graph $G$, if the representation number of $G$ is $k$, then every $k$-uniform word representing the graph $G$ is also square-free. Moreover, we prove that every minimal-length word representing a graph is square-free. Then, we count the number of possible square-free word-representations of a complete graph. At last, using the infinite square-free string generated from the Thue-Morse sequence, we prove that infinitely many square-free words represent a non-complete connected word-representable graph. 
    \\\textbf{Keywords:} word-representable graph, representation number, square-free word, Thue-Morse sequence.
	\end{abstract}

	\maketitle
	\pagestyle{myheadings}
	
\section{Introduction}
 Sergey Kitaev first introduced word-representable graphs based on the study of the celebrated Perkins semi-group \cite{kitaev2008word}. It is a very promising research area as the word-representable graphs generalize several key graph families, such as circle graphs, comparability graphs, 3-colourable graphs. As not all of the graphs are word-representable, finding graphs that are word-representable is an interesting problem. 

 Another interesting problem in the theory of word-representable graphs is finding the word $w$ that represents a graph such that $w$ contains some patterns or avoids some patterns. A square is two consecutive occurrences of a factor in a word. In our paper, we want to avoid the square in the word that represents a graph. The concept of square-free words attracted much more attention in the field of combinatorics on words since the work done by Axel Thue in his paper \cite{thue1906uber}. This work showed an infinite number of square-free words over the ternary alphabets and opened the area of combinatorics on words. The problem of finding square-free word-representation for the word-representable graph was proposed in \cite{kitaev2015words}, (\textit{Section 7.1.3}). There, it was shown that the trivial square can be avoided for all word-representable graphs except the empty graph of two vertices. Also, it was proved that for any word-representable graph, a cube-free word (a cube is a factor that occurs consecutively three times in a word) represents that graph. We want to analyze whether a non-trivial square-free word representation exists for each word-representable graph. This problem is an open problem in \cite{kitaev2015words} (\textit{Problem 7.1.10}). We positively answer this question in \textit{section \ref{sc2}} by showing that for connected word-representable graphs, after removing the squares from the word, the new word also represents the same graph. Also, for disconnected word-representable graphs, we give a square-free word that will represent that graph. After proving the existence of a square-free word for each word-representable graph, we want to count the number of square-free words representing a word-representable graph $G$. In \textit{section \ref{sc3}}, we count the number of possible square-free word-representations of the complete graph. Then, we use the infinite string generated from the Thue-Morse sequence \cite{thue1906uber} to create an infinite number of square-free words for non-complete connected word-representable graphs. 
 
 The rest of this section briefly describes all the required preliminary information on the word-representable graphs. 
 
 A \textit{subword} of a word $w$ is a word obtained by removing certain letters from $w$. In a word $w$, if $x$ and $y$ alternate, then $w$ contains $xyxy\cdots$ or $yxyx\cdots$ (odd or even length) as a subword.

\begin{definition}\textit{(\cite{kitaev2015words} , Definition 3.0.5)}.
 A simple graph $G = (V, E)$ is \textit{word-representable} if there exists a word $w$ over the alphabet $V$ such that letters $x$ and $y$ alternate in $w$ if and only if $xy \in E$, i.e., $x$ and $y$ are adjacent for each $x\not=y$. If a word $w$ \textit{represents} $G$, then $w$ contains each letter of $V(G)$ at least once.
\end{definition}

 In a word $w$ that represents a graph $G$, if $xy\notin E(G)$, then the non-alternation between $x$ and $y$ occurs if any one of these $xxy$, $yxx$, $xyy$, ${yyx}$ subwords is present in $w$.
\begin{definition}
    (\textit{\cite{kitaev2015words}, Definition 3.2.1.}) \textit{$k$-uniform word} is the word $w$ in which every letter occurs exactly $k$ times.
\end{definition}
\begin{definition} (\textit{\cite{kitaev2017comprehensive}, Definition 3})
    For a word-representable graph $G$, \textit{the representation number} is the least $k$ such that $G$ is $k$-representable.
\end{definition}
\begin{prop}(\textit{\cite{kitaev2015words}, Proposition 3.2.7})\label{uv}
		Let $w = uv$ be a $k$-uniform word representing a graph $G$, where $u$ and $v$ are two, possibly empty, words. Then, the word $w' = vu$ also represents $G$.
\end{prop}

The \textit{initial permutation} of $w$ is the permutation obtained by removing all but the leftmost occurrence of each letter $x$ in $w$, and it is denoted by $\pi(w)$. Similarly, the \textit{final permutation} of $w$ is the permutation obtained by removing all but the rightmost occurrence of each letter $x$ in $w$, and it is denoted $\sigma(w)$. For a word $w$, $w_{\{x_1, \cdots, x_m\}}$ denotes the word after removing all letters except the letters $x_1, \ldots, x_m$ present in $w$. 
 \begin{example}
     $w = 6345123215$, we have $\pi(w) = 634512$, $\sigma(w) = 643215$ and $w_{\{6,5\}} = 655$.
 \end{example}
	\begin{observation}(\textit{\cite{kitaev2008representable}, Observation 4})\label{pw}
		Let $w$ be the word-representant of $G$. Then $\pi(w)w$ also represents $G$.
	\end{observation}
 
 \begin{definition}\label{def1}
     For a $k$-uniform word $w$, the \textit{$i^{th}$ permutation} is denoted by $P_i$ where $P_i$ is the permutation obtained by removing all but $i^{th}$ occurrence of each letter $x$ in $w$. 
     
     We denote the $j^{th}$ occurrence of the letter $x$ in $w$ as $x_j$. 
 \end{definition}
 \begin{example}
    For word $w=142513624356152643$, $P_1=142536$, $P_2=124356$, $P_3=152643$.
 \end{example}
It can be easily seen that if $w$ is $k$-uniform, then $P_1=\pi(w)$ and $P_k=\sigma(w)$.
\begin{definition}(\textit{\cite{kitaev2015words}, Definition 3.2.8.}) 
A word $u$ contains a word $v$ as a \textit{factor} if $u = xvy$ where $x$ and $y$ can be empty words.
\end{definition}

\begin{example}
    The word $421231423$ contains the words $123$ and $42$ as factors, while all factors of the word $2131$ are $1$, $2$, $3$, $21$, $13$, $31$, $213$, $131$ and $2131$.
\end{example}
In the rest of the paper, we refer to the word-representable graphs as graphs. Also, $w=w_1w_2\cdots w_n$ denotes the word that contains $\{w_1,w_2,\ldots, w_n\}$ as factors where $w_i$ is a word possibly empty.
\noindent 
\section{Square free word-representation of a graph} \label{sc2}
In this section, we prove that besides the empty graph on two vertices, every word-representable graph has a square-free word-representation. The formal definition of square and the known results regarding square-free word-representation of a graph are described as follows.
\begin{definition}(\textit{\cite{kitaev2015words}, Definition 7.1.5})
  A \textit{square} in a word is two consecutive equal factors. A square of the form $xx$, where $x$ is a letter, is called \textit{trivial}.   
\end{definition}
For example, in $w=14251136536542$, $11$ and $365365$ are squares. Here, $11$ is a trivial square, and $365365$ is a non-trivial square. 
\begin{lemma}\textit{(\cite{kitaev2015words}, Lemma 7.1.7.)} 
If a word-representable graph G contains at least one edge, then its isolated vertices, if any, can be represented without involving trivial squares.
\end{lemma}
\begin{theorem}\label{tm1} \textit{(\cite{kitaev2015words}, Theorem 7.1.8.)}
  Trivial squares can be avoided in a word-representation of all but the graph $O_2$, the empty graph on two vertices. 
\end{theorem}
Therefore, for a graph $G$, there exists a word $w$ that represents $G$, and $w$ does not contain any trivial square. Moreover, in the theorem \ref{tm1}, the word $w=2134\cdots(n-1)12\cdots (n-1)n(n-1)\cdots21$  is used to represent the empty graph of order $n>2$. We can see that the word $w$ does not contain any nontrivial square. It was stated as an open problem whether a non-trivial square-free word $w$ exists for each word-representable graph. Here, we answer this question in the affirmative for word-representable non-empty graphs. 

To solve this problem, we prove some specific properties of the non-trivial squares present in a word that represents a connected graph. For vertices $x$ and $y$ in a graph $G$, $x\sim y$ denotes $x$ and $y$ are adjacent, and $x\nsim y$ denotes $x$ and $y$ are not adjacent.
\begin{lemma}\label{lm2}
    If $G(V, E)$ is a connected graph, and $w=uXXv$ represents $G$ where $X$ is a non-trivial square and $u$ and $v$ are factors of $w$, then $X$ contains every $x\in V$.
\end{lemma}
\begin{proof}
    Suppose there exists a vertex $x$ in $V$ which is not in $X$. As $G$ is connected, therefore, there exists a path from $x$ to $y$ for $y\in X$.
    Let, $x\sim x_1\sim x_2\sim \cdots \sim x_n\sim y$ be the path from $x$ to $y$.
    If $x_n\sim y$ and $x_n\notin X$, then $w=w_1x_nw_2w_3yw_4w_3yw_4w_5x_nw_6$ where $u=w_1x_nw_2$, $X=w_3yw_4$, $v=w_5x_nw_6$. As $x_n$ and $y$ are not alternating in $w$ implies $x_n\nsim y$ in $G$. But it is not possible, therefore $x_n\in X$.
    Suppose $\forall x_i$, $2\leq i\leq n$, $x_i \in X$. Now, $x_1\sim x_2$ implies $x_1 \in X$ because, if $x_1\notin X$ then $w=w_1x_1w_2w_3x_2w_4w_3x_2w_4w_5x_1w_6$, $u=w_1x_1w_2$, $X=w_3x_2w_4$, $v=w_5x_1w_6$. As $x_1$ and $x_2$ is not alternating so $x_1\nsim x_2$. 
    As $x\sim x_1$, using the same argument, we can show that $x\in X$. But it contradicts our assumption.
    Therefore, every vertex in $V$ also appears in $X$.
\end{proof}

\begin{lemma}\label{lm3}
    Suppose $G$ is a connected graph and $w=uXXv$ represents $G$, where $X$ is non-trivial square, if $x\sim y$, $\{x,y\}\in V(G)$ then $n(X_x)=n(X_y)$, where $n(X_x)$, $n(X_y)$ are the number of $x$ and $y$ present in $X$ respectively.
\end{lemma}
\begin{proof}
    Suppose $x\sim y$, then $x$ and $y$ alternates in $w$. So, $|n(X_x)-n(X_y)|\leq 1$ as else $x$ and $y$ cannot alternate in $X$. Suppose $|n(X_x)-n(X_y)|=1$. For $n(X_x)$ and $n(X_y)$, we consider the following cases:\\
    \textbf{Case 1}: In $X$, if $n(X_x)>n(X_y)$, then  $X_{\{x,y\}}=xyx\cdots x$. So, in $XX_{\{x,y\}}=xyx\cdots xxyx$ $\cdots x$,  but it is not possible because $x$ and $y$ alternates in $w$ implies $x$ and $y$ need to alternate in $XX$. Therefore, it contradicts our assumption.\\
    \textbf{Case 2}: In $X$, if $n(X_y)>n(X_x)$, then $X_{\{x,y\}}=yxy\cdots y$. We can show that it contradicts our assumption using the same argument as in \textit{case 1}.\\
    Hence, if $x\sim y$, then $n(X_x)-n(X_y)=0$. Therefore, $n(X_x)=n(X_y)$.
\end{proof}
Using these two lemmas, we prove that for a connected graph $G$, if $w$ is a word representing $G$ that contains squares, then after removing the squares in a certain way, the new square-free word also represents $G$.
\begin{theorem}\label{tm2}
    If $G$ is a connected graph and $w$ is a word representing $G$ where $w$ contains at least one square, then there exists a square-free word $w'$ that represents $G$.
\end{theorem}
\begin{proof}
    Suppose, $w=uXXv$, where $XX$ is the $1^{st}$ square present in $w$. According to Lemma \ref{lm2}, every vertex occurs at least once in the square $X$. If every $x$ occurs exactly once in $X$, then removing one $X$ represents the same graph. But if there exists $x$, $x\in X$ such that $x$ occurs more than once, then our claim is $w'=uX\pi(X)v$ also represents $G$ and does not contain the square $XX$. We need to check the following cases to prove that $w'=uX\pi(X)v$ also represents $G$.\\
    \textbf{Case 1}: If $x\sim y$, then they are alternating in $X$ as $x,y\in X$. From the Lemma \ref{lm3}, we know $n(X_x)=n(X_y)$. Without loss of generality, we assume $X_{\{x,y\}}=xy\cdots xy$, then $\pi(X)_{\{x,y\}}=xy$. So, $x$ and $y$ alternate in $X\pi(X)$. Therefore, using $X\pi(X)$ instead of $XX$ do not affect the alternation of $x$ and $y$.\\
    \textbf{Case 2}: If $x\nsim y$, then they are not alternating in $w$. Without loss of generality, suppose that there exists a subword $xxy$ in $w$. Then, the following disjoint cases can occur.
    \\\textbf{Case 2.1}: If $xxy \in u_{\{x,y\}}$ or $xxy \in v_{\{x,y\}}$, then replacing $XX$ with $X\pi(X)$ does not remove the subword $xxy$ . 
    \\\textbf{Case 2.2}: If $xxy\in X$, then $w=uw_1xw_2xw_3yw_4w_1xw_2xw_3yw_4v$, $X=w_1xw_2xw_3yw_4$. So, removing one $X$ from $w$ does not remove the subword $xxy$. So, we can replace $XX$ with $X\pi(X)$. \\
    \textbf{Case 2.3}: If $xxy \notin X$, then the only way for $xxy\in XX$ is when $w=uw_1xw_2yw_3xw_4w_1x$ $w_2yw_3xw_4v$,  $X=w_1xw_2yw_3xw_4$. Then $\pi(X)_{\{x,y\}}=xy$ implies $X\pi(X)_{\{x,y\}}=xy\cdots x xy$. So, $xxy$ factor appear in $X\pi(X)_{\{x,y\}}$ that is $x$ and $y$ do not alternate in $X\pi(X)$. \\
    \textbf{Case 2.4}: If $xxy\in uX$, then replacing $XX$ with $X\pi(X)$ does not remove $xxy\in uX$.\\
    \textbf{Case 2.5}: If $xxy\in Xv$, then $w=uXw_1xw_2xw_3w_4yw_5$, $X=w_1xw_2xw_3$, $\{x,y\}\notin w_2\cup w_3\cup w_4$ or $w=uXw_1xw_2w_3xw_4yw_5$, $X=w_1xw_2$, $\{x,y\}\notin w_2\cup w_3\cup w_4$ are the two possible cases. So, for these two cases, we need to check whether $X\pi(X)$ holds the non-alternation.\\
    \textbf{Case 2.5.1}: If $w=uXw_1xw_2xw_3 w_4yw_5$, where $X=w_1xw_2xw_3$, $v=w_4yw_5$ then  $y$ is present in $w_1$. Because, according to Lemma \ref{lm2}, $y\in X$ and according to our assumption only $w_1$ can contain $y$ in $X$. So, $yxx$ is present in $1^{st}$ $X$. \\
    \textbf{Case 2.5.2}: If $w=uXw_1xw_2w_3xw_4yw_5$, where $X=w_1xw_2$, $v=w_3xw_4yw_5$ then if $x$ occurs once in $X$ then $X\pi(X)v$ also contain $xxy$. If $x$ occurs more than once, then let ${w_1}_{\{x,y\}}=xyxy\cdots xy$. Therefore, $XX_{\{x,y\}}=xyxy\cdots xyxxy\cdots xyx$, where $X_{\{x,y\}}=xyxyxy\cdots xyx$ because according to our assumption, $y\notin w_2$. So, it implies that $X\pi(X)_{\{x,y\}}=xyxy\cdots xyxxy$. So, the $xxy$ factor remains in the new word. 
    
    We can remove $XX$ from $w$ using the above approaches. After removing the square $XX$ from $w=uXXv$, the new word does not create any square with the factors present in $u$. So, replacing $XX$ with $X\pi(X)$ may create only $X\pi(X)\pi(X)$ square in the new word. But, as mentioned earlier, we can eliminate one of the $\pi(X)$. Therefore, no extra squares are created in the word when we remove the squares. If more than one square exists in the word $w$, we apply these processes for all those squares. So, after removing all of these squares, we get a word $w'$, which does not contain any square and also represents graph $G$.
\end{proof}
 As there exists a square-free word for all word-representable graphs, we want to verify for a word-representable graph $G$ having representation number $k$, whether a $k$-uniform word that represents $G$ is square-free.
\begin{lemma}\label{lm4}
    If $G$ is a connected word-representable graph and the representation number of $G$ is $k$, then every $k$-uniform word representing $G$ is square-free.
\end{lemma}
\begin{proof}
Suppose there exists a word $w=uXXv$ which is $k$-uniform and represents graph $G$. Based on the number of occurrences of letters in the square, we consider the following cases.
\\ \textbf{Case 1}: If every letter occurs once in $X$, then using the square removal process mentioned in Theorem \ref{tm2}, we can construct a word whose representation number is less than $k$ and still represents the graph $G$. But it is a contradiction as the representation number of $G$ is $k$. 
\\ \textbf{Case 2}: If every letter occurs more than once in $X$ then using the same argument mentioned in \textit{case 1}, we can contradict our assumption.
\\ \textbf{Case 3}: If there exists some letters $x_1,x_2,\ldots,x_i$, $1\leq i<|V(G)|$ occurs once and there exists some letters $y_1,y_2,\ldots,y_j$, $1\leq i<|V(G)|$ occurs more than once. Then, according to Lemma \ref{lm3}, there does not exist any edge between $x_i$ and $y_i$ as the number of occurrences of them is not equal. Therefore, $x_1, x_2,\ldots, x_i$ and $y_1, y_2,\ldots,y_j$ are parts of different disconnected components. But, it contradicts our assumption about connected graph $G$. Therefore, every vertex can occur once in the square, or if any vertex occurs more than once, then every vertex needs to occur more than once. But both cases are discussed in \textit{case 1} and \textit{case 2}. In both of the cases, we find a contradiction. Therefore, $w$ does not contain any square.
\end{proof}

 We want to check whether the minimal length word representing a graph is square-free.

\begin{lemma}\label{lm5}
     For a word-representable graph $G$, every minimal-length word representing $G$ is square-free.
\end{lemma}
\begin{proof}
    Suppose there exists a minimal length word $w=uXXv$ whose length is $l$, and it represents graph $G$. Using the square removal process mentioned in Theorem \ref{tm2}, we can construct a word whose length is less than $l$ and still represents the graph $G$. But it is a contradiction as the minimal length of the word representing $G$ is $l$.
\end{proof}

Now, we know we can get a square-free word-representation for a connected graph $G$. But, if $G$ is not connected, and the connected components are word-representable, then using the square-free word-representation of these components, we derive a square-free word-representation of $G$. As we already know about the square-free word for the empty graph of vertices more than $2$, here, we consider at least one component to be a non-empty word-representable graph. The term $w\setminus l(w)$ denotes the word $w$ after removing the last letter $l(w)$ of the word $w$. For example, $w=1235463125$, then $w\setminus l(w)=123546312$.  
\begin{theorem}\label{tm3}
    If $G$ is a disconnected graph, and $G_i$, $1\leq i\leq n$, $n\in \mathbb{N}$ are the connected components of $G$ and $w_i$ is the square-free word-representation of $G_i$ and $G_1$ is a non-empty word representable graph, then the word $w=w_1\setminus l(w_1)w_2\cdots w_nl(w_1)\sigma(w_n)\cdots\sigma(w_2)$ $\sigma(w_1)\setminus l(w_1)\sigma(w_2)\cdots\sigma(w_n)l(w_1)$, where $l(w_1)$ is the last letter of the word $w_1$, represents $G$ and $w$ is a square-free word. 
\end{theorem}
\begin{proof}
  For every $x\in V(G_i)$ and $y \in V(G_j)$ , $i\neq j$, $1\leq i\leq n$, $1\leq j\leq n$, $x$ and $y$ should not alternate in $w$. For every $G_i$, $i>1$, as subword $w_2w_3\cdots w_n \sigma(w_n)\cdots \sigma(w_3)\sigma(w_2)$ is present in $w$, so $\forall x \in V(G_i)$, $\forall y \in V(G_j)$ , $i\neq j$ , $x$ and $y$ do not alternate in $w$. In the word $w$, the subword $w_1\setminus l(w_1)w_2w_3\cdots w_n\sigma(w_n)\cdots \sigma(w_3)\sigma(w_2)$ $\sigma(w_1)\setminus l(w_1)$ is present. Therefore, $\forall x \in V(G_1)\setminus l(w_1)$, $\forall y\in G_i$, $x$ and $y$ do not alternate in $w$.
    Now, in the word $w$, $l(w_1)w_n\cdots w_3w_2\sigma(w_2)\sigma(w_3)\cdots \sigma(w_n) l(w_1)$ is a subword. So, $l(w_1)$ does not alternate with $y$, $\forall y\in G_i$, $2\leq i\leq n$. Therefore, $w$ represents $G$.
    
    There does not exist any square of the form $w_iw_j$ in $w$ where $w_i$ and $w_j$ represent the components $G_i$ and $G_j$ respectively. And as, every $G_i$, $1\leq i\leq n$, $n\in \mathbb{N}$ are connected components and $w_i$ is a square-free word-representation of $G_i$ so $w$ does not contain any square.
\end{proof}
From Theorem \ref{tm3}, we get a square-free word $w$ for disconnected graphs. We want to find a condition for reducing the occurrence of the letters in the square-free word $w$ such that the new word is still square-free. In the following, we show that if all graph components are complete graphs, then it is impossible.
\begin{corollary}
    If $G$ is a disconnected graph, and all of its connected components $G_i$, $1\leq i\leq n$, $n\in \mathbb{N}$ are complete graphs, then the square-free word-representation of $G$ is $3$-representable. Moreover, there does not exist a square-free $2$-uniform word $w$ that represents the graph $G$.
\end{corollary}
\begin{proof}
    In Theorem \ref{tm3}, if all $G_i$'s are $1$-representable then the corresponding $w_i$'s occur $3$ times in $w$ and in each $w_i$, $V(G_i)$ occurs only once. So, from Theorem \ref{tm3}, it can be shown that for $G$, there exists a $3$-uniform word $w$ which is square-free.

    Suppose there exists a word $w$ such that $w$ is a square-free 2-uniform word representing the graph $G$. Also, let each $w_i$ start with $i$, and $N_i$ is the neighbour of $i$. Without loss of generality, we assume $w$ starts with the letter $1$. As, $w$ is 2-uniform, therefore, $w_{\{1,N_1\}}=1N_11N_1$. Now, $\forall v\in V(G)\setminus \{1,N_1\}$, $v$ needs to occur $2$ times between two $1$'s, or after the two $1$'s, or else it alternates with $1$. These conditions are also true for all $u\in N_1$. So, $v$ must occur $2$ times between two $u$'s or before or after the two $u$'s. But if $vv$ occurs after $N_1$, then $w=1N_11N_1vv$ and $w$ contains a square $1N_11N_1$. Therefore, the only possible values of $w$ are $w=1N_1vv1N_1$ or $w=1vvN_11N_1$. So, it contains $vv$ as a factor.
  
 Now we consider the vertex $2$ and its neighbours $N_2$. We can use the same argument for $v'\in V(G)\setminus \{1,2, N_1, N_2\}$ and claim that $v'v'$ also occurs as a factor. We can extend this argument up to the vertex $n-1$ and its neighbour $N_{n-1}$. For the remaining graph $G_n$, the vertex $n$ and its neighbour $N_n$ need to occur $2$ times between two $n-1$'s or after the two $n-1$'s. Also, $\forall y\in N_{n-1}$, $n$ and $N_n$ need to occur $2$ times between two $y$'s, or after the two $y$'s. As $G_n$ is a complete graph, $P_iP_i$, where $P_i$ is an arbitrary permutation of the vertices of $G_n$ is the representation possible between two $n-1$'s, or after the two $n-1$'s. But it is not a square-free word. So, it contradicts the assumption.

\end{proof}
But, for a disconnected graph having word-representable components, whether we can find a shorter square-free word than the word $w$ described in Theorem \ref{tm3} if there exists a component that is not a complete graph. We show that a shorter square-free word exists rather than the word $w$.
\begin{corollary}
    For a disconnected graph $G$, whose connected components are $G_i$, $1\leq i\leq n$, $n\in \mathbb{N}$, and $w_i$ is the square-free word-representation of $G_i$, if there exists a connected component $G_j$ that is not a complete graph whose representation number is $k$, then $w= w_1\cdots w_{j-1}w_jw_{j+1}\cdots w_n\sigma(w_j)$ $\sigma(w_n)\cdots \sigma(w_{j+1})$ $\sigma(w_{j-1})$ $\cdots \sigma(w_1)$ represents $G$, $w_j$ is a $k$-uniform word representing the graph $G_j$.
\end{corollary}
\begin{proof}
   As $G_j$ is not complete, therefore, $k\geq 2$. $w_j$ is a $k$-uniform word. According to Lemma \ref{lm4}, $w_j$ is square-free. In the word $w$, $\forall x\in V(G_j)$, $\forall y\notin V(G_j)$, $x$ and $y$ do not alternate because of the factor $w_1\cdots w_{j-1}w_j\cdots w_n$. And, $w_{\{V(w_1),V(w_2),\ldots ,V(w_{j-1}),V(w_{j+1}),\ldots V(w_n)\}}=w_1\cdots w_{j-1}\cdots w_n$ $\sigma(w_n)\cdots \sigma(w_{j+1})\sigma(w_{j-1})\cdots \sigma(w_1)$. So, there is no alternation between $x\in V(G_i)$ and $y\in V(G_{i+1})$. By Theorem \ref{tm2}, each $G_i$ has a square-free word, so clearly, we can say that there are no squares (trivial and non-trivial) in $w$ and $w$ represents $G$.
\end{proof}
We proved that every word-representable graph can be represented by a square-free word except $O_2$. We know that from a word $w$ that represents a graph, we can create another word $\pi(w)w$ that also represents that graph. So, an infinite number of words exist that represent a word-representable graph. But, we want to check whether this is true for square-free words representing a word-representable graph. In the next section, we show that an infinite number of square-free words exist that represent a non-complete word-representable graph.

\section{Number of square-free words representing a graph} \label{sc3}
In this section, we want to count the number of square-free words representing a graph. We start with the complete graph $K_n$ of $n$ vertices. We find out that for $K_n$, a finite number of square-free word representations exists, and we prove it in the following.

\begin{lemma}\label{lm16}
    If $w$ is a word that represents the complete graph $K_n$ and the length of $w$ is greater than $2n-1$, then $w$ contains at least a square. 
\end{lemma}
\begin{proof}
    Suppose $w$ is a square-free word of length $2n$, and $w$ represents the $K_n$ graph. Let, $\{1,2,3,\ldots,n\}$ be the vertices of $K_n$. Without loss of generality, we assume $w$ starts with the letter $1$. Then, $1$ cannot occur twice until every $i$, $2\leq i\leq n$ occurs once. Because if $1$ occurs twice before some $i$, then $i$ and $1$ do not alternate in $w$. But, for a complete graph, each vertex should alternate with one another. Without loss of generality, we assume the first $n$ letters of $w$ are $123\cdots n$. Now, our claim is the last $n$ letters of $w$ are also $123\cdots n$. Suppose our claim is not true, then we consider the following cases.
    \\ \textbf{Case 1}: For some $i\in V(K_n)$, $i$ occurs more than once in the last $n$ letters of the word $w$. But, from the $n$ number of vertices of $K_n$, we choose the last $n$ letters of $w$. So, if $i$ occurs more than once in the last $n$ letters, then there exists some $j \in V(K_n)$, $i\neq j$ such that $j$ does not occur in the last $n$ letters of $w$. Then, $i$ and $j$ does not alternate in $w$. But, as we mentioned earlier, it is not possible.
   \\ \textbf{Case 2}: For some $\{i,j\}\in V(K_n)$, $i\neq j$, the first $n$ letters of $w$ is $123\cdots i\cdots j\cdots n$ but in last $n$ letters of $w$ is $123\cdots j \cdots i \cdots n$. Then, $w_{\{i,j\}}=ijji$ implies $i$ and $j$ are not alternating. But, as we mentioned earlier, it is not possible.
\\Therefore, the last $n$ letters of $w$ is also $123\cdots n$. Then, $w=123\cdots n123\cdots n$, which is square. But, it contradicts our assumption about the square-free word $w$. Therefore, $w$ is a square-free word representing $K_n$ if the length of $w$ is less than $2n$.
\end{proof}
From the Lemma \ref{lm16}, we find that the length of the square-free words representing a complete graph is less than $2n$. Now, we can count the number of square-free words representing a complete graph. In the following theorem, we count the number of possible square-free word-representations of a complete graph.
\begin{theorem}
    For a complete graph $K_n$, the number of square-free word which represents $K_n$ is $n\times n!$.
\end{theorem}
\begin{proof}
    According to Lemma \ref{lm16}, the length of a square-free word-representation of a complete graph $K_n$ is less than $2n$, and the length of a word-representation of a graph of $n$ vertices is greater than $n-1$. Let, $V=\{1,2,3,\ldots,n\}$ be the vertices of $K_n$. We know that any permutation of $V$ is a word that represents $K_n$. Let $P$ be an arbitrary permutation of $V$. Without loss of generality, we assume $P=123\cdots n$. Now, we can extend this word $P$ up to $2n-1$ length word representing $K_n$. We can extend this word $P$ in the following manner.
    \\ \textbf{Case 1}: We can add new letters after the word $P$. In the proof of Lemma \ref{lm16}, we can see that if $P=123\cdots i\cdots j \cdots n$, $\{i,j\}\in V$, $i\neq j$, then we are only able to add letter $j$ after $i$ in the extension of $P$. Also, we can extend $P$ to $P1$, $P12$, $\cdots$, $P123\cdots i\cdots j \cdots n-1$. There does not exist any extension of $P$, where $w=P12\cdots i-ji$, $j<i-1$ and $w$ represents $K_n$. Because, $w_{\{i-j+1,i\}}=i-j+1ii$ implies $i-j+1$ and $i$ does not alternate in $w$. But, every letter should alternate with one another in the word $w$. Therefore, from $P$, we can construct $n$ number of square-free words (including $P$). So, there exists $n!$ number of permutations that represent the graph $K_n$. Therefore, we can construct $n\times n!$ number of square-free words representing the graph $K_n$.
    \\ \textbf{Case 2}: We can add new letters before the word $P$. If $P=123\cdots i\cdots j \cdots n$, $\{i,j\}\in V$, $i\neq j$, then we can add letter $i$ before $j$ in the extension of $P$. If we add $i$ after $j$ in the extension, then there exists a $jiij$ subword in the extension. But, $i$ and $j$ should alternate in the extension. So, $i$ can not appear after the $j$ in the extension of $P$. Also, as mentioned in \textit{case 1}, if $i$ and $j$ are added to the new word, then every letter after $j$ and every letter between $i$ and $j$ also appear in the new word. We claim that every word we can create using this process is already considered in \textit{case 1}. Suppose $w= i\cdots j\cdots n-1n123\cdots i-1i\cdots j \cdots n$ is a word not included in \textit{case 1}. But, in $w$, $P_1=i\cdots j\cdots n-1n123\cdots i-1$ is a permutation of $V$, and $i\cdots j \cdots n$ is the extension we add after $P_1$. This case is already considered in \textit{case 1}. Therefore, it contradicts our assumption. So, adding a new letter in front of a permutation is already included in \textit{case 1}.
    \\ \textbf{Case 3}: We can add new letters before and after the word $P$. But, from \textit{case 1} and \textit{case 2}, we know that for adding a new letter in $P$, we need to maintain the same order of the letters present in the word $P$. If, $P=123\cdots l\cdots m\cdots i\cdots j \cdots n$, then we can create a word $i\cdots j \cdots n123\cdots l\cdots m\cdots i-1i\cdots j \cdots n123\cdots l\cdots m$ and it also represents the graph $K_n$. But, we can see that $P_2=i\cdots j \cdots n123\cdots l$ $\cdots m\cdots i-1$ is a permutation of $V$ and $i\cdots j \cdots n123\cdots l\cdots m$ is the extension we add after $P_2$. So, this case is also considered in \textit{case 1}.
    \\Therefore, the number of square-free words representing the graph $K_n$ is $n\times n!$.
    
\end{proof}
For complete graphs, we count the number of all possible square-free word-representations. Now, we can try to count the number of all possible square-free word-representations of the other word-representable graphs. For other word-representable graphs, we describe a way to create more square-free word-representations from the infinite square-free string derived from the Thue-Morse sequence. The formal definitions and theorems regarding the Thue-Morse sequence are described below.

\begin{definition}\cite{shallit2009asecond}
Let, for $n\geq 1$, $n\in \mathbb{N}$
 \begin{equation*}
      t_n=
    \begin{cases}
        0, & \text{if the number of $1$'s in the base-2 expansion of $n$ is even;}\\
        1, & \text{if the number of $1$'s in the base-2 expansion of $n$ is odd;}
    \end{cases}  
    \end{equation*}
    Then, $t=t_1t_2t_3\cdots$, this infinite string $t$ is called the \textit{Thue-Morse sequence}.
\end{definition}
\begin{theorem} \textit{(\cite{shallit2009asecond}, Theorem 2.5.2.)}\label{tm8}
    For $n\geq 1$,  define $c_n$ to  be  the  number  of $1$'s  between  the $n^{th}$  and $(n+1)^{th}$  occurrence of $0$ in  the string $t$. Set $c=c_1c_2c_3\cdots$. Then $c=210201\cdots$ is an infinite square-free string over the alphabet $\{0,1,2\}$.
\end{theorem}

   In \textit{Definition \ref{def1}}, we defined $P_i$ is the $i^{th}$ permutation in a word $w$. The infinite square-free word is over three letters, so we can replace each letter with a permutation $P_i$. But after replacing each letter, does the new word represent the same graph and remain square-free? For that, we derive some properties of the permutations present in a word $w$ that represents a graph $G$. All the graphs we have considered here are assumed to be connected word-representable graphs.
\begin{lemma}\label{lm6}
      If the representation number of a graph $G$ is $k$, $k>2$ and $w$ is a $k$-uniform word representing the graph $G$, then each $P_i$, $1\leq i\leq k$, is unique in $w$.   
\end{lemma}
\begin{proof}
    Suppose there exist $P_i$, $1\leq i\leq k$ and $P_j$ for $1\leq j\leq k$, such that $P_i=P_j$, $i\neq j$. Now, our claim is if we remove $P_j$ from $w$, then the new word $w'$ also represents $G$. 

    For proving $w'$ representing graph $G$, we are considering the following cases for two arbitrary vertices $x$ and $y$ of the graph $G$. Without loss of generality, we assume $i<j$.
    \\
    \textbf{Case 1}: If $x\sim y$, then $x$ and $y$ should alternate in $w'$. We know that $x$ and $y$ are alternating in $w$. Without loss of generality, assuming $x$ occurred before $y$. Therefore, $x_1<y_1<x_2<y_2<\cdots<x_i<y_i<\cdots <x_j<y_j<\cdots<x_k<y_k$ where $x_i$ and $y_i$ are the $i^{th}$ occurrence of $x$ and $y$ respectively in $w$. So, in $w'$, $x_j$ and $y_j$ term is removed, then it becomes $x_1<y_1<x_2<y_2<\cdots<x_i<y_i<\cdots <x_{j-1}<y_{j-1}<x_{j+1}<y_{j+1}<\cdots<x_k<y_k$. From this, we see that the alternation is preserved in $w'$.
    \\
    \textbf{Case 2}: If $x\nsim y$, then $x$ and $y$ should not alternate in $w'$. Without loss of generality we assume, $x_1<y_1<x_2<y_2<\cdots<y_l<x_l<\cdots <x_k<y_k$, where $x_i$ and $y_i$ are the $i^{th}$ occurrence of $x$ and $y$ respectively in $w$ and $k>2$. If $k=2$, then only $i^{th}$ and $j^{th}$ occurrences of $x$ and $y$ is possible to present in $w_{\{x,y\}}$. But after eliminating any of the $i^{th}$ and $j^{th}$ occurrences from $w_{\{x,y\}}$,we obtain only the $xy$ term. But, $x$ and $y$ cannot alternate in $w_{\{x,y\}}$. Now, if $l\neq i$ or $l\neq j$ then the $l^{th}$ occurrence of $x$ and $y$ is $y_l<x_l$ while the other occurrences of $x$ and $y$ are $x<y$. So, $w'$ preserves the non-alternation of $x$ and $y$. If $l=j$ then  $x_1<y_1<x_2<y_2<\cdots<y_j<x_j<\cdots<x_k<y_k$. But as $P_i=P_j$, so in the $i^{th}$ occurrence of $x$ and $y$, $y_i<x_i$ occurs. So, after removal of $P_j$, the occurrence of $x$ and $y$ is $x_1<y_1<x_2<y_2<\cdots<y_i<x_i<\cdots<x_k<y_k$. As the non-alteration of $x$ and $y$ are preserved in $w'$, $w'$ also represents the graph $G$. But $w'$ is a $k-1$-uniform word, which is a contradiction as the representation number of $G$ is $k$. Therefore, each $P_i$ is unique in $w$.
\end{proof}
For a word $w$, the terms $s(w)$ and $l(w)$ represent the first and last letter of $w$ respectively. For example, in the word $w=1243142$, $s(w)=1$ and $l(w)=2$. In the following lemma, we prove that for any two permutations derived from a word, the starting letter of one permutation is not equal to the last letter of the other.  
\begin{lemma}\label{lm7}
 If $w$ is a k-uniform word representing a graph $G$, then for all $P_i$ and $P_j$, $i\neq j$, $l(P_i)\neq s(P_j)$ and for any $P_i$, $s(w)\neq l(P_i)$, $l(w)\neq s(P_i)$. 
\end{lemma}
\begin{proof}
 Suppose  there exist two permutations $P_i$ and $P_j$, $i<j$ such that $l(P_i)=s(P_j)$. Let, $x=l(P_i)=s(P_j)$, then in $w$, the occurrence of $x$ is $x_1<x_2<\cdots<x_i<\cdots<x_j<\cdots<x_k$, where $x_i$ denote the $i^{th}$ occurrence of $x$. Now, let $N_x$ be the neighbour of $x$, then $x$ and $N_x$ should alternate in $w$. There are two cases for the occurrence of $x$ and $N_x$.\\
 \textbf{Case 1}: If $(N_x)_1<x_1$, then for each $i^{th}$ occurrence of $N_x$ and $x$ in $w$ is $(N_x)_i<x_i$ . Therefore, the occurrence of $N_x$ and $x$ is $(N_x)_1<x_1<(N_x)_2<x_2<\cdots<(N_x)_i<x_i<(N_x)_{i+1}<\cdots<(N_x)_j<x_j<(N_x)_{j+1}<\cdots<x_k$. Here, $(N_x)_j<x_j $ implies some other letters present before $x$ in the $j^{th}$ occurrence of the letters in $w$. But it contradicts our assumption, as $s(P_j)=x$.
 \\
 \textbf{Case 2}: If $(N_x)_1>x_1$, then for each $i^{th}$ occurrence of $N_x$ and $x$ in $w$ is $(N_x)_i>x_i$. Therefore, the occurrence of $N_x$ and $x$ is $x_1<(N_x)_1<x_2<(N_x)_2<\cdots<(N_x)_{i-1}<x_i<(N_x)_i<\cdots<(N_x)_{j-1} <x_j<(N_x)_j <\cdots<x_k$. Here, $(N_x)_i>x_i $ implies some other letters present after $x$ in the $i^{th}$ occurrence of the letters in $w$. But it contradicts our assumption that $l(P_i)=x$.
 Therefore, for all $P_i$ and $P_j$, $i\neq j$, $l(P_i)\neq s(P_j)$.

 Now, we know $s(w)=s(P_1)$ and $l(w)=l(P_k)$. If there exist a $P_i$ such that $s(w)=l(P_i)$ or $l(w)= s(P_i)$ then $s(w)=l(P_i)=s(P_1)$ or $l(w)= s(P_i)=l(P_k)$. But, it is not possible. Therefore, for any $P_i$, $s(w)\neq l(P_i)$, $l(w)\neq s(P_i)$. 
\end{proof}

 Now, according to the Lemma \ref{lm7}, for a word-representable graph $G$, having representation number $k\geq 3$, there are at least $3$ permutations labelled as $P_1$, $P_2$ and $P_3$ where none of them is same. Now, using these three permutations, we define the following homomorphism $\{P_1, P_2, P_3\}\rightarrow \{0,1,2\}$ for the square-free string $c$ mentioned in Theorem \ref{tm8}.
  \begin{equation*}
      h(x)=
    \begin{cases}
        P_1, & \text{if } x=2\\
        P_2, & \text{if } x=1\\
        P_3, &\text {if } x=0
    \end{cases}  
    \end{equation*}
    Now, $w'=h(c)$ where $P_1$, $P_2$ and $P_3$ are three different permutations obtained from a $k$-uniform word $w$ representing graph $G$.

    For a word $w$, $w[l]$ denotes the word obtained by restricting the word $w$ up to length $l$ from the initial position. For example, $w=154631467351423$, $w[5]=15463$ and $w[8]=15463146$. We use this notation in the following theorem.
\begin{theorem}\label{tm4}
    For a word-representable graph $G$, having representation number $k\geq 3$, if $w$ is the $k$-uniform word representing $G$ then, $ww'[n\times i]$ also represents the graph $G$, where $n=|V(G)|$, $i\in \mathbb{N}$ and $w'=h(c)$. 
\end{theorem}
\begin{proof}
    $w'$ is the word where each letter of the square-free string $c$ is replaced with a permutation. From Lemma \ref{lm6} and Lemma \ref{lm7} it can be seen that there exist at least $3$ permutations $P_1$, $P_2$ and $P_3$ which are unique and $l(P_i)\neq s(P_j)$, where $i\neq j$, $1\leq i\leq 3$, $1\leq j\leq 3$. So, in the $P_iP_j$ factor, no trivial square exists. So, the following two cases need to be considered for $ww'[n\times i]$ to represent the graph $G$.
    \\ \textbf{Case 1}: If $x\sim y$ in $G$, then $x$ and $y$ should alternate in $ww'[n\times i]$. As, $w$ represents $G$, without loss of generality we assume the occurrences of $x$ and $y$ in $w$ is $x_1<y_1<x_2<y_2<\cdots<x_k<y_k$. So, in each $P_i$, $1\leq i\leq k$, the $i^{th}$ occurrences of $x$ and $y$ is $x_i<y_i$. Suppose we are taking the three permutation $P_a$, $P_b$ and $P_c$, where $a\neq b\neq c$ and $1\leq a\leq k$, $1\leq b\leq k$, $1\leq c\leq k$. In these three permutations, occurrences of $x$ and $y$ is $x_a<y_a$, $x_b<y_b$ and $x_c<y_c$. Hence, $x$ and $y$ alternate in the $w'[n\times i]$ implies $x$ and $y$ alternate in $ww'[n\times i]$.
    \\\textbf{Case 2}: if $x \nsim y$ in $G$, then $x$ and $y$ should not alternate in $ww'[n\times i]$. As, $w$ represents $G$, therefore non-alternation of $x$ and $y$ is preserved in $w$. So, non-alternation of $x$ and $y$ is preserved in $ww'[n\times i]$.
    Hence the word $ww'[n\times i]$ also represents $G$.
\end{proof}
We are using the word $w$ that represents a graph $G$ in $ww'[n\times i]$, and we will prove that this word is a square-free word. But, if the last letters of $w$ are a permutation of the vertices, then it may create a square with $w[n]$, as $w[n]$ is also a permutation. But, we claim that after removing at most $n=|V(G)|$, there does not exist any permutation in the last letters of $w$. 
\begin{lemma}\label{lm8}
    For a word-representable graph $G$ having the representation number $k$, $k>1$ and the $k$-uniform word $w$ representing $G$, if $w=w_1v$, where $v$ is the last letters of $w$ such that removing $v$ removes the permutation of the last letters then $|v|\leq n$ where $|v|$ is the length of $v$.
    \end{lemma}
    \begin{proof}
        Suppose for the word $w$, $|v|>n$. Let, $\{v_1,v_2,\cdots,v_n\}$ be the vertices of $G$. As $|v|>n$, at least each $v_i$ occurs once in $v$, or no permutation exists. Without loss of generality, we assume the last $n$ letters of $w$ are $v_1v_2\cdots v_n$. Now, after removing $v_n$, if there is still a permutation, then $w=w_2v_nv_1v_2\cdots v_n$ else there does not exist any permutation. Then removing $v_{n-1}$ also maintains the permutation implies $w=w_3v_{n-1}v_nv_1v_2\cdots v_{n-1}v_n$. If we continue this process up to $v_2$ then using the same argument, we can show that $w=w_4v_2\cdots v_nv_1v_2\cdots v_n$. As $|v|>n$, we must remove $v_1$. But then $w=w_5v_1v_2\cdots v_nv_1v_2\cdots v_n$, which creates a contradiction because according to Lemma \ref{lm4}, $w$ is a square-free word. Therefore $|v|\leq n$.
    \end{proof}
As we know, $ww'[n\times i]$ is also a word-representation of $G$, so we use this technique to create a $k+i$ uniform square-free word representing the graph $G$. The construction is as follows:
\begin{enumerate} 
    \item If last $n$ letters of $w$ is not create a permutation then we concatenate $w'[n\times i]$ with $w$.
    \item Else $w=w_1v$, where $v$ is last letters of $w$ such that removing $v$ remove the permutation of the last letters. From the Lemma \ref{lm8},  we know that$|v|\leq n$, according to Proposition \ref{uv}, $w_2=vw_1$ also represents $G$ and $w_2$ do not have any permutation in the last $n$ letter. This new word is also $k$-uniform, so it is square-free. Now, we concatenate $w'[n\times i]$ with the new word.  
\end{enumerate}
\begin{theorem}\label{tm5}
    For a word-representable graph $G$, having representation number $k\geq 3$, the word $ww'[n\times i]$, $i\in \mathbb{N}$ constructed using the above technique, is a square-free representation of $G$, where $w$ is a $k$-uniform square-free word-representation of $G$, and $w'=h(c)$. 
\end{theorem}
\begin{proof}
    As $w$ is a $k$-uniform word, according to Lemma \ref{lm4}, $w$ is a square-free word. Now, $P_1$, $P_2$, and $P_3$ are unique permutations present in $w$, and there does not exist any square in any of them. From the square-free string $c$, we can say that in $w'[n\times i]$, there exists no square of form $P_iP_i$. Also, from the construction of $ww'[n\times i]$, we already avoid the occurrence of permutation in the last $n$ letters of $w$. So, there does not exist any square of form $P_iP_i$ in $ww'[n]$. Therefore, for the following cases, a square may occur.\\
    \textbf{Case 1}: If there exists a trivial square $xx$ in $P_iP_j$, $i\neq j$, $1\leq i\leq 3$, $1\leq j\leq 3$, but according to Lemma \ref{lm7}, $l(P_i)\neq s(P_j)$. So, no trivial square exists in $P_iP_j$. \\
    \textbf{Case 2}: If there exist some $XX$ factors in $P_iP_j$ such that $P_iP_j=uXXv$, $u$ or $v$ is a nonempty string. We know $ww'[n\times i]$ represents the graph $G$, and according to Lemma \ref{lm2}, all vertices are present in $X$. Every vertex occurs only once in $P_i$ and $P_j$, so $XX=P_iP_j$. But it implies that square-free string $c$ has a square. Therefore, there does not exist any squares in $ww'[n\times i]$.
\end{proof}

Now, applying the same technique used in the word $ww'[n\times i]$, we can create more square-free words from $2$-uniform word-representation of a word-representable graph. But we get only two permutations for a $2$-uniform word $w$. We prove for $k>2$, we can find unique permutations in $k$-uniform words. But, for $k=2$, if $P_1=P_2$, then we can perform a cyclic shift of that word and derive a new word, where $P_1\neq P_2$. We prove this statement in the following.
\begin{lemma}\label{lm17}
If $G$ is a connected word-representable graph, and the representation number of $G$ is $2$, then there exists a word $w$ such that $\pi(w)\neq \sigma(w)$ and $w$ represents the graph $G$.
\end{lemma}
\begin{proof}
Suppose $w$ is a $2$-uniform word that represents the graph $G$, and 
$\pi(w)=\sigma(w)$. We assume that, $w=(P_1)_1(P_2)_1(P_1)_2(P_2)_2\cdots 
(P_1)_i(P_2)_j$, $2\leq i\leq n$, $2\leq j\leq n$, and $P_1=(P_1)_1(P_1)_2\cdots
(P_1)_i=\pi(w)$, $P_2=(P_2)_1(P_2)_2\cdots (P_2)_j=\sigma(w)$. If $i=j=1$, then 
$w=P_1P_2$. But $P_1=P_2$, so $w$ represents the complete graph, and the complete
graph is $1$-represenatable. Let, $l$ be the last letter of $ (P_2)_j$, then $l$ 
also be the last letter of $(P_1)_i$. According to Proposition \ref{uv}, 
$w'=l(P_1)_1(P_2)_1(P_1)_2(P_2)_2\cdots (P_1)_i(P_2)_j\setminus l$ is also 
represent the graph $G$. In the word $w'$, $\pi(w')$ start with $l$ but 
$\sigma(w')$ cannot start with $l$ because, $l\notin (P_2)_1$. Therefore, in the word $w'$, $\pi(w')\neq \sigma(w')$ and $w'$ represents the graph $G$.
\end{proof}
According to Lemma \ref{lm17}, we can find a word $w$ such that $\pi(w)\neq \sigma(w)$. So, we define $P_1$ and $P_2$ as $\pi(w)$ and $\sigma(w)$ respectively, and $P_3$ as $w$ in the word $ww'[n\times i]$.

\begin{corollary}\label{tm6}
    The word $ww'[n\times i]$ for a 2-uniform word-representable graph $G$ is also square-free if $P_1$ and $P_2$ are the $\pi(w)$ and $\sigma(w)$ respectively and $P_3$ is the word $w$ that represents the graph $G$.
\end{corollary}
\begin{proof}
    According to Lemma \ref{lm7}, there does not exist any trivial square between any two of $P_1$, $P_2$ and $P_3$ because $l(P_i)\neq s(P_j)$, $i\neq j$, $1\leq i\leq 3$, $1\leq j\leq 3$. Then, using the same argument in Theorem \ref{tm5}, we can prove that $ww'[n\times i]$ is square-free. 
\end{proof}
If we get a $k$-uniform square-free word, we can create the word $ww'[n\times i]$ and extend the $k$-uniform word to $(k+1)$-uniform word. Also, from the Theorem \ref{tm5} and Corollary \ref{tm6}, we know that $ww'[n\times i]$ is square-free. Therefore, if $G$ is a non-complete word-representable connected graph, then there exists an infinite number of square-free words that represent the graph $G$.

\section{Conclusion}
We prove that for any simple word-representable graph $G$, there exists a square-free word $w$ representing $G$. We count the number of possible square-free word-representations of a complete graph. Moreover, we also introduce a way to create arbitrary long square-free words representing a simple word-representable graph. We list some open problems and directions for further research related to these topics below.
\begin{enumerate}
    \item The number of square-free words representing a non-complete connected graph is shown to be infinite. Will it be the same for disconnected graphs?  
    \item  We remove the $XX$ pattern from a word $w$, which still represents the same graph. There exist other patterns in a word like an abelian square ($XP(X)$, $P(X)$ is a permutation of $X$), border ($XuX$) etc. Will removing those patterns still represent the same graph?
    \item If after removing different patterns from a word, that word still represents the same graphs, then count the number of such pattern-free words possible that represent the same graph.

    \end{enumerate}

	\bibliographystyle{plain}
	\bibliography{ref.bib}

\begin{thebibliography}{1}

\bibitem{kitaev2017comprehensive}
Sergey Kitaev.
\newblock A comprehensive introduction to the theory of word-representable graphs.
\newblock In {\em International Conference on Developments in Language Theory}, pages 36--67. Springer, 2017.

\bibitem{kitaev2015words}
Sergey Kitaev and Vadim Lozin.
\newblock {\em Words and graphs}.
\newblock Springer, 2015.

\bibitem{kitaev2008representable}
Sergey Kitaev and Artem Pyatkin.
\newblock On representable graphs.
\newblock {\em Journal of automata, languages and combinatorics}, 13(1):45--54, 2008.

\bibitem{kitaev2008word}
Sergey Kitaev and Steve Seif.
\newblock Word problem of the perkins semigroup via directed acyclic graphs.
\newblock {\em Order}, 25(3):177--194, 2008.

\bibitem{shallit2009asecond}
Jeffrey Shallit.
\newblock {\em A second Course in Formal Languages And Automata Theory}.
\newblock Cambridge University Press, 2009.

\bibitem{thue1906uber}
Axel Thue.
\newblock Uber unendliche zeichenreihen.
\newblock {\em Norske Vid Selsk. Skr. I Mat-Nat Kl.(Christiana)}, 7:1--22, 1906.

\end{thebibliography}
\end{document}